\documentclass[intlimits,12pt,reqno]{amsart}
\usepackage{latexsym,amsthm,amsmath,amssymb,mathrsfs,mathtools}
\usepackage[T1]{fontenc}
\usepackage[utf8]{inputenc}
\usepackage[mathscr]{eucal}
\usepackage{enumerate}
\usepackage{enumitem}
\usepackage{soul}

\usepackage{hyperref}
\usepackage{xcolor}
\hypersetup{
    colorlinks=true,
    linkcolor=blue,
    citecolor=blue,
    urlcolor=blue
}

plus 6pt minus 12pt
\def\mvint_#1{\mathchoice
          {\mathop{\vrule width 6pt height 3 pt depth -2.5pt
                  \kern -9pt \intop}\limits_{\kern -3pt #1}}%
          {\mathop{\vrule width 5pt height 3 pt depth -2.6pt
                  \kern -6pt \intop}\nolimits_{#1}}%
          {\mathop{\vrule width 5pt height 3 pt depth -2.6pt
                  \kern -6pt \intop}\nolimits_{#1}}%
          {\mathop{\vrule width 5pt height 3 pt depth -2.6pt
                  \kern -6pt \intop}\nolimits_{#1}}}

\newcommand{\MM}{\mathcal M}

\newcommand{\HH}{\mathcal H}

\newcommand{\bbbr}{\mathbb R}

\newcommand{\bbbs}{\mathbb S}

\newcommand{\N}{\mathbb N}

\newcommand{\R}{\mathbb R}
\newcommand{\overbar}[1]{\mkern 1.7mu\overline{\mkern-1.7mu#1\mkern-1.5mu}\mkern 1.5mu}

\newcommand{\eps}{\varepsilon}

\def\dist{\operatorname{dist}}

\newtheorem{theorem}{Theorem}[section]
\newtheorem*{theorem*}{Theorem}
\newtheorem{lemma}[theorem]{Lemma}
\newtheorem{corollary}[theorem]{Corollary}
\newtheorem{proposition}[theorem]{Proposition}

\theoremstyle{definition}

\newtheorem{remark}[theorem]{Remark}
\newtheorem*{remark*}{Remark}

\usepackage{setspace}

\begin{document}

\sloppy

\title[Differentiability of convex functions]{A geometric approach to second-order differentiability of convex functions}

\author{Daniel Azagra}
\address{Departamento de An{\'a}lisis Matem{\'a}tico y Matem\'atica Aplicada,
Facultad Ciencias Matem{\'a}ticas, Universidad Complutense, 28040, Madrid, Spain.  
}
\email{azagra@mat.ucm.es}

\author{Anthony Cappello}
\address{Department of Mathematics,
University of Pittsburgh,
301 Thackeray Hall,
Pittsburgh, PA 15260, USA.}
\email{arc172@pitt.edu }

\author{Piotr Haj\l asz}
\address{Department of Mathematics,
University of Pittsburgh,
301 Thackeray Hall,
Pittsburgh, PA 15260, USA.}
\email{hajlasz@pitt.edu}
\thanks{P.H. was supported by NSF grant  DMS-2055171 and 
by Simons Foundation grant
917582.}



\begin{abstract}
We show a new, elementary and geometric proof of the classical Alexandrov theorem about the second order differentiability of convex functions. We also show new proofs of recent results about Lusin approximation of convex functions and convex bodies by $C^{1,1}$ convex functions and convex bodies. 
\end{abstract}

\maketitle

\section{Introduction}
\label{S1}
The aim of this paper is to provide a new, elementary and geometric proof of the following classical theorem of Alexandrov
\cite{Alexandroff}.
For a history of the theorem and a list of known proofs, see~\cite{colesanti}.
\begin{theorem}
\label{T7}
If $f:\R^n\to\R$ is convex, then it is differentiable a.e. and at almost every point where $f$ is differentiable, there is a symmetric matrix denoted by $D^2f(x)$ such that
\begin{equation}
\label{eq3}
\lim_{y\to x} 
\frac{f(y)-f(x)-Df(x)(y-x)-\frac{1}{2}(y-x)^T D^2f(x)(y-x)}{|y-x|^2}=0.
\end{equation}
\end{theorem}

Our proof is so simple and geometric in nature that its concept can be described in just a few sentences.
All notation used in the Introduction will be explained in Section~\ref{S2}.

Consider the set $W(\delta)$, the union of all closed balls of radius $\delta>0$ contained in the epigraph of $f$. The set $W(\delta)$ is convex and it is the epigraph of a convex function $g$. Clearly, $g\geq f$. 
Using elementary and geometric arguments we show that if $\delta>0$ is sufficiently small, then the set $W(\delta)$ touches the graph of $f$ along a set that is large in the sense of measure. More precisely,
for every $R>0$ and every $\eps>0$, there is $\delta>0$ such that
$$
|\{x\in B^n(0,R):\, f(x)\neq g(x)\}|<\eps.
$$
Since the convex set $W(\delta)$ is the union of balls of fixed radius, it is well known and easy to prove that the boundary of $W(\delta)$ is of class $C^{1,1}_{\rm loc}$, and hence $g\in C^{1,1}_{\rm loc}$, i.e., the gradient of $g$ is locally Lipschitz continuous. 
Since $g\in C^{1,1}_{\rm loc}$, it follows from the Rademacher theorem that $g$ is twice differentiable almost everywhere in the classical sense. Now it remains to observe that at almost all points $x$ such that $f(x)=g(x)$, $f$ is twice differentiable in the sense of \eqref{eq3}. Namely, this is true whenever $x$ is a density point of the set $\{f=g\}$ and $g$ is twice differentiable at $x$. In that case \eqref{eq3} is satisfied with $D^2f(x):= D^2 g(x)$.

There is also a second version of the Alexandrov theorem which says that the subdifferential $\partial f$ is differentiable a.e. 
\begin{theorem}
\label{T8}
If $f:\R^n\to\R$ is convex, then for all $x\in\R^n$ where $f$ is twice differentiable as in \eqref{eq3}, we have 
\begin{equation}
\label{eq13}
\lim_{y\to x}\sup_{\sigma_y\in\partial f(y)}\frac{|\sigma_y-Df(x)-D^2f(x)(y-x)|}{|y-x|}=0.
\end{equation}
\end{theorem}
The usual way to prove Theorem~\ref{T7} is to show Theorem~\ref{T8} first and conclude Theorem~\ref{T7} from it. 
In our approach we will prove Theorem~\ref{T7} directly and we will conclude Theorem~\ref{T8} as a corollary.

The argument described above also leads to new and elementary proofs of the following recent results \cite[Theorem~1.4, Corollary~1.7, Theorem~1.12 and Corollary~1.13]{AzagraH}.
\begin{theorem}
\label{T2}
Let $f:\R^n\to\R$ be a convex function. Then, for every measurable set $A\subset\R^n$ of finite Lebesgue measure, and for every $\varepsilon>0$, there exists a convex function $g\in C^{1,1}(\R^n)$ such that
$$
|\{x\in A:\, f(x)\neq g(x)\}|<\varepsilon.
$$
\end{theorem}
\begin{theorem}
\label{T1}
Let $K$ be a convex body in $\R^n$. Then
for every $\eps>0$, there is a convex body $W\subset K$, such that $\partial W\in C^{1,1}$ and 
$$
\HH^{n-1}(\partial K\triangle\partial W)<\eps.
$$
In fact, there is $\delta_o>0$ such that for every $\delta\in (0,\delta_o)$, the set $W$ defined as the union of all closed balls of radius $\delta$ that are contained in $K$ satisfies the claim of the theorem.
\end{theorem}
A {\em convex body} is a compact convex set $K\subset\R^n$ with non-empty interior.
Notation $A\triangle B$ stands for the symmetric difference of the sets $A$ and $B$, that is, 
$
A\triangle B :=(A\setminus B)\cup (B\setminus A),
$
and $\HH^s$ denotes the $s$-dimensional Hausdorff measure.
\begin{theorem}
\label{global theorem}
Let $f:\R^n\to\R$ be a convex function, and assume that $f\not\in C^{1,1}_{{\rm loc}}(\R^n)$. Then the following conditions are equivalent:
\begin{enumerate}
\item For every $\eps>0$ there exists a convex function $g\in C^{1,1}_{\rm loc}(\R^n)$ such that 
$$
|\{x\in \R^n : f(x)\neq g(x)\}|<\varepsilon.
$$ 
\item The function $f$ is essentially coercive. 
\end{enumerate}
Moreover, if  $f$ is essentially coercive, we can find $g$ satisfying $g\geq f$.
\end{theorem}
We call a convex function $f:\R^n\to\R$  {\em essentially coercive} if there exists a linear function $\ell:\R^n\to\R$ such that 
$\lim_{|x|\to\infty}\left( f(x)-\ell(x)\right)=\infty$; this is equivalent to saying that the epigraph of $f$ does not contain lines, see \cite[Theorem~1.11]{AM2}. Here and in what follows by a line we mean a set isometric to $\R$ so half-line is not a line.
\begin{theorem}
\label{corollary for convex hypersurfaces}
Let $S$ be a convex hypersurface of $\R^n$, and assume that $S$ is not of class $C^{1,1}_{{\rm loc}}$. Then the following assertions are equivalent:
\begin{enumerate}
\item For every $\varepsilon>0$ there exists a convex hypersurface $S_{\varepsilon}$  of $\R^n$ of class $C^{1,1}_{\rm loc}$  such that
$\mathcal{H}^{n-1}\left(S\triangle S_{\varepsilon}\right)< \varepsilon$.
\item $S$ does not contain any line.
\end{enumerate}
\end{theorem}
We call the boundary $\partial W$ of a closed convex set $W$ with nonempty interior (not necessarily bounded) a {\em convex hypersurface}, and we say that it is of class $C^{1,1}_{\rm loc}$ if it is locally a graph of a $C^{1,1}$ function (if the set $W$ is unbounded, we will say that $W$ is an {\em unbounded convex body}). 
\begin{remark}
\label{R2}
It follows from the proof that if $S=\partial W$, where $W$ is a (possibly unbounded) convex body that contains no lines, then there exists a (possibly unbounded) convex body $W_{\varepsilon}\subset W$ such that $S_\eps:=\partial W_\eps\in C^{1,1}_{\textrm{loc}}$ and $\mathcal{H}^{n-1}\left(S\triangle S_{\varepsilon}\right)< \varepsilon$.
\end{remark}

From Theorem~\ref{corollary for convex hypersurfaces} we will also deduce the following  new generalization of Theorem~\ref{global theorem} for convex functions defined on arbitrary open convex subsets of $\R^n$.

\begin{theorem}
\label{global theorem for arbitrary U}
Let $U\subset\R^n$ be open and convex, and let $f:U\to\R$ be a convex function, such that $f\not\in C^{1,1}_{\rm loc}(U)$. Then, the following statements are equivalent:
\begin{enumerate}
\item For every $\varepsilon>0$ there exists a convex function $g\in C^{1,1}_{\rm loc}(U)$  such that 
\begin{equation}
\label{eq23}
|\{x\in U : f(x)\neq g(x)\}|<\varepsilon.
\end{equation}
\item The graph of $f$ does not contain any line of $\R^{n+1}$.
\end{enumerate}
Moreover, if the graph of $f$ contains no lines, we can find $g$ satisfying $g\geq f$.
\end{theorem}
\begin{remark}
It was recently proved in \cite{ADH} that if $f:U\to\R$ is locally strongly convex, then there is a locally strongly convex function $g\in C^2(U)$ that satisfies \eqref{eq23} (and other estimates). The proof is however, much more difficult.
\end{remark}

The original proofs of Theorems~\ref{T2},~\ref{T1} and~\ref{global theorem} used the Whitney extension theorem for convex functions \cite{AGM,AM,Azagra2}, and the Alexandrov Theorem~\ref{T7}. Our proofs presented here are elementary and avoid these tools. As explained above, the proofs are based on a simple geometric idea that is also used in our proof of Alexandrov's theorem.

We will prove Theorem~\ref{T1} first and we will use it as a main tool in the proofs of Theorems~\ref{T7} and~\ref{T2}. Indeed, the brief description of the proof of Theorem~\ref{T7} presented above is based on the approximation of the epigraph of $f$ by the convex set $W(\delta)$ of class $C^{1,1}_{\rm loc}$ and this is strictly related to Theorem~\ref{T1}.

Except Section~\ref{S7},
our exposition is elementary and self-contained. We have made an effort to make it accessible to anyone with basic knowledge of real analysis, and no knowledge in convex analysis is required.

The paper is structured as follows.
In Section~\ref{S2} we fix notation and recall basic definitions and facts needed to understand the paper. All results mentioned in this section are well known.
In Section~\ref{S3} we prove Theorem~\ref{T1} and then we use it to prove Corollary~\ref{T14} which is a version of Theorem~\ref{T2}.
This corollary will play a central role in the proofs of Theorems~\ref{T7}, ~\ref{T2} and ~\ref{global theorem}.
Theorems~\ref{T7} and~\ref{T2} are proved in Sections~\ref{S4} and~\ref{S6} respectively. In Section~\ref{S5} we prove Theorem~\ref{T8} as a direct consequence of Theorem~\ref{T7}. This proof is independent of all other sections of the paper and it can be read independently. In Section~\ref{S7} we present the proofs of Theorems~\ref{global theorem}, \ref{corollary for convex hypersurfaces}, and~\ref{global theorem for arbitrary U}.

We made an effort to make different parts of the paper as independent as possible. Section~\ref{S3} is needed in Sections~\ref{S4},~\ref{S6} and~\ref{S7}, but the content in Sections~\ref{S4},~\ref{S6} and~\ref{S7} of the paper stands alone and is not dependent on one another. Similarly Section~\ref{S5} is independent of any other part of the paper.

\section{Preliminaries} 
\label{S2}
In this brief section we will explain notation and basic facts needed in the paper. This section will also clarify necessary prerequisites. By no means the definitions and facts presented here are detailed. The reader may find missing details in standard textbooks. 

Balls in $\R^n$ are denoted by $B(x,r)$ or $B^n(x,r)$. The unit sphere in $\R^n$ that is centered at the origin is denoted by $\bbbs^{n-1}$. The interior of a set $A$ is denoted by $\operatorname{int} A$. An interval in $\R^n$ with endpoints $x,y\in\R^n$ is denoted by $[x,y]$.
The scalar product of vectors $u,v\in\R^n$ is denoted by $\langle u,v\rangle$.

The Lebesgue measure of $A\subset\R^n$ is denoted by $|A|$. We say that $x\in \R^n$  is a {\em density point} of a measurable set $A\subset\R^n$ if $\frac{|A\cap B(x,r)|}{|B(x,r)|}\to 1$ as $r\to 0^+$. It follows from the Lebesgue differentiation theorem that almost all points $x\in A$ are density points of $A$. 

The Hausdorff measure is denoted by $\HH^s$. It follows from the definition that if $f$ is $L$-Lipschitz, then $\HH^s(f(A))\leq L^s\HH^s(A)$. If $A\subset\R^n$, $\lambda>0$ and $\lambda A=\{\lambda x:\, x\in A\}$ is the dilation of $A$ by the factor $\lambda$, then $\HH^s(\lambda A)=\lambda^s\HH^s(A)$. $\HH^n$ coincides with the Lebesgue measure in $\R^n$.

We say that $f\in C^{1,1}(U)$ ($f\in C^{1,1}_{\rm loc}(U)$), if $U\subset\R^n$ is open, $f\in C^1(U)$, and $D f$ is Lipschitz (locally Lipschitz) continuous on $U$.
If $f\in C^{1,1}(B^n(0,R))$, then it follows that
\begin{equation}
\label{eq14}
|f(y)-f(x)-Df(x)(y-x)|\leq M|y-x|^2
\quad
\text{for all } x,y\in B^n(0,R),
\end{equation}
where $M$ is the Lipschitz constant of $Df$. Indeed, we can write $f(y)-f(x)=Df(\xi)(y-x)$ for some $\xi\in [x,y]$ and \eqref{eq14} follows. This inequality implies that if $f\in C^{1,1}_{\rm loc}(U)$, where $U\subset\R^n$ is open, then
\begin{equation}
\label{eq15}
f(y)=f(x)+Df(x)(y-x)+O(|y-x|^2)
\quad
\text{for all } x,y\in U.
\end{equation}
We say that the boundary of a bounded domain $U\subset\R^n$ is of class $C^{1,1}$ if it is locally a graph of a $C^{1,1}$ function.

We  use notation $\nabla f(x)$ for the gradient vector while $Df(x)$ is the linear derivative. With this notation we have $Df(x)v=\langle\nabla f(x),v\rangle$.

If $W\subset\R^n$ is a closed convex set, then it is easy to see that for every $x\in\R^n$, there is a unique point denoted by $\pi_W(x)$ such that
\begin{equation}
\label{eq17}
\pi_W(x)\in W
\qquad
\text{and}
\qquad
|x-\pi_W(x)|=\dist(x,W).
\end{equation}
Clearly, if $x\not\in W$, then $\pi_W(x)\in\partial W$. 
The next result is well known see e.g., \cite[Proposition~3.1.3]{HUL} or \cite[Theorem~1.2.1]{schneider}.
\begin{lemma}
\label{T15}
$\pi_W:\R^n\to W$ is $1$-Lipschitz.
\end{lemma}

The {\em convex hull} of a set $A\subset\R^n$ (defined as the intersection of all convex sets containing $A$, or equivalently, as the set of all convex combinations of points of $A$) is denoted by $\operatorname{co}(A)$.
Every closed and convex set $W\subset\R^n$ is the intersection of all closed half-spaces that contain $W$. In fact,  for every $x\in \partial W$ there is a half-space $H_x$ such that $W\subset H_x$ and $x\in T_x\cap W$, where $T_x=\partial H_x$. The hyperplane $T_x$ is called a {\em hyperplane supporting} $W$ at $x$. Thus for every $x\in \partial W$, there is a hyperplane supporting $W$ at $x$, but such a hyperplane is not necessarily unique. This implies that if $U\subset\R^n$ is open and convex and $f:U\to\R$ is convex, then for every $x\in U$ there is $v\in\R^n$ such that $f(y)\geq f(x)+\langle v,y-x\rangle$ for all $y\in U$. 
Indeed, on the right hand side we have an equation of the supporting hyperplane of the convex {\em epigraph} $\operatorname{epi}(f)=\{(x,y)\in U\times\R:\, x\in U,\ y\geq f(x)\}$.
The set of all such $v$ is denoted by $\partial f(x)$ and called the {\em subdifferential} of $f$ at $x$. Thus $\partial f(x)\neq\varnothing$ for any $x\in U$. If in addition $f$ is differentiable at $x_o$, then $\partial f(x_o)=\{\nabla f(x_o)\}$ 
i.e., $f(y)\geq f(x_o)+Df(x_o)(y-x_o)$ meaning that the tangent hyperplane to the graph of $f$ at $x_o$ is the unique hyperplane supporting the epigraph of $f$ at $(x_o,f(x_o))$. Convex functions are locally Lipschitz continuous and hence they are differentiable a.e. by the Rademacher theorem, so $\partial f(x)=\{\nabla f(x)\}$ for almost all $x\in U$. In fact we will prove the a.e. differentiability of convex functions directly and {\em without} any reference to  the Rademacher theorem, see Corollary~\ref{T4} and Remark~\ref{R1}, but we will need the Rademacher theorem in the proof of Theorem~\ref{T7}, because we will need to know that the gradient of a convex function $g\in C^{1,1}$ is differentiable a.e.

\section{Proof of Theorem~\ref{T1}}
\label{S3}
We will precede the proof with auxiliary results.

For a convex body $K\subset\R^n$ and $r>0$ we define the {\em inner parallel convex body} by
$$
K_r:=\{x\in K:\, \dist(x,\partial K)\geq r\}.
$$
\begin{lemma}
\label{T18}
$K_r$ is convex for any $r>0$.
\end{lemma}
\begin{proof}
Let $x,y\in K_r$. We need to show that $[x,y]\subset K_r$. Clearly, 
$\overbar{B}(x,r),\overbar{B}(y,r)\subset K$ and for any $z\in [x,y]$,
$
\overbar{B}(z,r)\subset \operatorname{co}(\overbar{B}(x,r)\cup\overbar{B}(y,r))\subset K,
$
so $\dist(z,\partial K)\geq r$, $z\in K_r$, and hence $[x,y]\subset K_r$. 
\end{proof}
Let $r_o=\sup_{x\in K}\dist(x,\partial K)$. Clearly $K_r=\varnothing$ for $r>r_o$.
$K_{r_o}\neq\varnothing$, but it has empty interior. However, for $r\in (0,r_o)$, $K_r$ has non-empty interior, so $K_r$ is a convex body only for $r\in (0,r_o)$. 

For a convex body $K$ and $r>0$ we also define
\begin{equation}
\label{eq1}
K(r):=\bigcup\{\overbar{B}(x,r):\, \overbar B(x,r)\subset K\}.
\end{equation}
It is easy to see that $K(r)$ is convex and compact (it can be empty). Moreover, if $K$ contains a ball of radius $r_o$, then for any $r\in (0,r_o]$, $K(r)$ has non-empty interior and hence $K(r)$ is a convex body.
\begin{lemma}
\label{T5}
If a convex body $K$ contains a ball of radius $r_o$, then for all $r\in (0,r_o)$, $K_r$ is a convex body, and
\begin{equation}
\label{eq2}
\HH^{n-1}(\partial K_r)\leq \HH^{n-1}(\partial K\cap\partial K(r)).
\end{equation}
\end{lemma}
\begin{proof}
Clearly, for $r\in (0,r_o)$, $K_r$ has non-empty interior, so it is a convex body by Lemma~\ref{T18}.
Observe that (see \eqref{eq17})
\begin{equation}
\label{eq7}
\pi_{K_r}(\partial K\cap \partial K(r))=\partial  K_r.
\end{equation}
Indeed, if $z\in \partial K_r$, then there is $x\in\partial K$, such that $|x-z|=r$. Therefore,
$x\in\overbar{B}(z,r)\subset K$, and hence $x\in K(r)$. Thus,
$x\in\partial K\cap\partial K(r)$,
$|x-z|=r\geq\dist(x,K_r)$, and hence $z=\pi_{K_r}(x)$.
Now, \eqref{eq2} follows from \eqref{eq7} and the fact that $\pi_{K_r}$ is $1$-Lipschitz (Lemma~\ref{T15}).
\end{proof}
The next beautiful result is due to McMullen~\cite{mcmullen}. 
While it can be concluded from 
Alexandrov's theorem, we present here a direct and surprisingly elementary proof which is a small modification of McMullen's argument. 
In fact, Lemma~\ref{T6} will play an important role in our proof of Alexandrov's theorem.
\begin{lemma}
\label{T6}
If $K\subset\bbbr^n$ is a convex body, then 
$\lim_{r\to 0^+}\HH^{n-1}(\partial K\setminus \partial K(r))=0.$
\end{lemma}
\begin{remark}
Lemma~\ref{T6} has the following geometric interpretation: for almost all $x\in\partial K$, there is a closed ball $\overbar{B}\subset K$ touching the boundary of $K$ at $x$, i.e., $x\in\overbar{B}$.
\end{remark}
\begin{proof}
Without loss of generality we may assume that
$\overbar{B}(0,r_o)\subset K$. If $r\in (0,r_o)$, then $0$ belongs to the interior of $K_r$. For $\lambda >0$ we define
$$
\lambda K_r:=\{\lambda z :\, z\in K_r\},
$$
that is, $\lambda K_r$ is a dilation of $K_r$.
For $r\in (0,r_o)$, let
$$
\lambda(r):=\inf\{\lambda>0:\, K\subset \lambda K_r\}.
$$
Clearly, $K\subset\lambda(r)K_r$. It is easy to see that the function $r\mapsto\lambda(r)$ is non-decreasing and $\lambda(r)\to 1$ as $r\to 0^+$. Indeed, for any $\eps>0$, $(1+\eps)^{-1}K\subset \operatorname{int}K$, and hence
$\delta:=\dist((1+\eps)^{-1}K,\partial K)>0$, so for all $r\in (0,\delta]$
$$
(1+\eps)^{-1}K\subset K_r,
\qquad
\text{i.e.,} 
\qquad 
K\subset (1+\eps)K_r.
$$
In other words $1\leq\lambda(r)\leq 1+\eps$ for all $0<r\leq\delta$ proving that $\lambda(r)\to 1$ as $r\to 0^+$.

It is easy to see that $\pi_K(\partial(\lambda(r)K_r))=\partial K$ (see \eqref{eq17}). Indeed, if $x\in\partial K$ and $\nu(x)$ is the outer unit normal vector to a supporting hyperplane of $K$ at $x$, then there is $t\geq 0$ such that $z:=x+t\nu(x)\in\partial(\lambda(r)K_r)$ and it easily follows that $\pi_K(z)=x$. Since $\pi_K$ is $1$-Lipschitz and it maps $\partial(\lambda(r)K_r)$ onto $\partial K$, we have that
\[
\begin{split}
\HH^{n-1}(\partial K)
&\leq 
\HH^{n-1}(\partial(\lambda(r)K_r))=
\lambda(r)^{n-1}\HH^{n-1}(\partial K_r)\leq
\lambda(r)^{n-1}\HH^{n-1}(\partial K\cap \partial K(r))\\
&\leq 
\lambda(r)^{n-1}\HH^{n-1}(\partial K) \to \HH^{n-1}(\partial K)
\quad
\text{as } r\to 0^+.
\end{split}
\]
Therefore, $\HH^{n-1}(\partial K\cap\partial K(r))\to\HH^{n-1}(\partial K)$, as $r\to 0^+$.
This completes the proof of Lemma~\ref{T6}.
\end{proof}
\begin{lemma}
\label{T19}
Let $f,g:B^n(0,R)\to\R$ be convex functions. If $g\in C^{1,1}$, $f\leq g$ and $f(x)=g(x)$ for some $x\in B^n(0,R)$, then $f$ is differentiable at $x$, $Df(x)=Dg(x)$ and 
\begin{equation}
\label{eq10}
f(y)=f(x)+Df(x)(y-x)+O(|y-x|^2).
\end{equation}
\end{lemma}
\begin{proof}
If $v\in\partial f(x)$, then clearly, $v\in\partial g(x)$ and hence $v=\nabla g(x)$. Therefore, the result follows from the estimate
$$
f(x)+\langle\nabla g(x),y-x\rangle\leq f(y)\leq g(y)=f(x)+\langle\nabla g(x),y-x\rangle+O(|y-x|^2),
$$
where in the last equality we used \eqref{eq15} and the fact that $g(x)=f(x)$.
\end{proof}
\begin{corollary}
\label{T4}
If $f:\bbbr^n\to \R$ is convex, then it is differentiable a.e. Moreover
\begin{equation}
\label{eq4}
f(y)=f(x)+Df(x)(y-x)+O(|y-x|^2) 
\qquad
\text{for almost all $x\in\R^n$.}
\end{equation}
\end{corollary}
\begin{proof}
Since the boundary of a ball is parameterized by a smooth convex function, Lemma~\ref{T19} implies \eqref{eq4} whenever there is a ball in the epigraph of $f$ that touches the graph of $f$ at $(x,f(x))$ and it follows from Lemma~\ref{T6} that it is true for almost all $x$.
\end{proof}
\begin{remark}
\label{R1}
Note that the proof of Corollary~\ref{T4} does not use Rademacher's theorem. Moreover, the estimate \eqref{eq4}, is stronger than the a.e. differentiability of $f$ that would follow from an application of Rademacher's theorem.
We will not need Corollary~\ref{T4} in this paper.
\end{remark}

The following result, was proven in a more general form in the unpublished work \cite[Theorem~1, p.\ 32]{lucas}. It is also mentioned without any proof or reference in \cite{kiselman}. Although a detailed proof can be found in \cite[Proposition~2.4.3]{Hormander}, the origin of the result is not referenced in this work.
\begin{lemma}
\label{T3}
A convex body $W$ has $C^{1,1}$ boundary if and only if there is $r>0$ such that $W=W(r)$.
\end{lemma}
\begin{remark}
In other words a convex body $W$ has boundary of class $C^{1,1}$ if and only if there is $r>0$ such that $W$ is the union of closed balls of radius $r$.
\end{remark}
We will only prove the implication from right to left, that is we will prove that if $W=W(r)$, then $\partial W$ is of class $C^{1,1}$. This is the only implication that we need in the proof of Theorem~\ref{T1}. For the proof of the implication from left to right, see \cite[Proposition~2.4.3]{Hormander}. We will present two proofs. The first proof is sketched only and it uses the implicit function theorem. The second one is detailed and it does not use the implicit function theorem.
\begin{proof}[First proof]
It is very elementary and easy to prove that if $K\subset\R^n$ is compact and convex, then the function $d^2_K(x)=\dist(x,K)^2$ is differentiable and $\nabla d^2_K(x)=2(x-\pi_K(x))$, see \cite[p. 181]{HUL}. Since the function $\pi_K$ is Lipschitz by Lemma~\ref{T15}, we have that $d_K^2\in C^{1,1}$ and all points in $\R^n\setminus K$ are regular so for $t>0$, $\{x:\,\dist(x,K)=t\}=(d_K^2)^{-1}(t^2)$ is a $C^{1,1}$-submanifold of $\R^n$ and hence it is locally a graph of a $C^{1,1}$ function by the implicit function theorem. It remains to observe that $\partial W(r)=\{x\in \R^n:\, \dist(x,W_r)=r\}$.    
\end{proof}
\begin{proof}[Second proof]
Thus, we assume that for each $p\in\partial W$ there is $h(p)\in W$ such that $p\in\overbar{B}(h(p),r)\subset W$. It follows that the hyperplane $T_p$ tangent to the ball $\overbar{B}(h(p),r)$ at $p$ is the unique hyperplane supporting $W$ at $p$. 

Note that $\dist(h(p),\partial W)=r$ implies that $h(p)\in W_r$, so $|p-h(p)|=r=\dist(p,W_r)$, and hence $h(p)=\pi_{W_r}(p)$ (see \eqref{eq17}).

The inner unit normal vector to $T_p$ is given by
$$
\nu(p)=\frac{h(p)-p}{r}=\frac{\pi_{W_r}(p)-p}{r}
$$
and Lemma~\ref{T15} implies that the function $\nu:\partial W\to \bbbs^{n-1}$ is Lipschitz continuous:
$$
|\nu(p)-\nu(q)|\leq\frac{|\pi_{W_r}(p)-\pi_{W_r}(q)|+|p-q|}{r}\leq\frac{2}{r}|p-q|.
$$
This in turn, implies that the boundary $\partial W$ is of class $C^{1,1}$. Indeed, choose any point $p_o\in\partial W$ and choose a Euclidean coordinate system $(x_1,\ldots,x_n)=(x',x_n)$ such that $p_o=0$ and $T_{p_o}=\{x_n=0\}$. Then $\partial W$ in a neighborhood $U=B^{n-1}(0,\frac{r}{2})$ of $p_o=0$ is a graph of a function $x_n=f(x')$ i.e., $p(x'):=(x',f(x'))\in\partial W$. Since for $x'\in U$, the graph of $f$ lies above $T_{p(x')}$ and below $\overbar{B}(h(p(x')),r)$, it follows from geometric considerations 
(as in the proof of Corollary~\ref{T4}) that $f$ is differentiable at $x'$ and $T_{p(x')}$ is the tangent hyperplane to the graph of $f$ at $p(x')$. 
Note also that $|\nabla f|\leq M$ on $U$ for some $M>0$, because the tangent hyperplane to the graph of $f$ cannot intersect with $B(h(0),r)$.
It remains to show that $\nabla f$ is Lipschitz continuous in $U$.

The inner unit normal vector in terms of $\nabla f$ is given by
$$
\nu(p(x'))=\frac{(-\nabla f(x'),1)}{\sqrt{1+|\nabla f(x')|^2}},
\qquad
\text{so}
\qquad
\pi(\nu(p(x'))=\frac{-\nabla f(x')}{\sqrt{1+|\nabla f(x')|^2}},
$$
where $\pi:\R^{n}\to\R^{n-1}$, $\pi(x',x_n)=x'$ is the orthogonal projection. 
Since $\Psi(\Phi(z))=z$ for all $z\in\R^{n-1}$, where
$\Psi(z)=-z/\sqrt{1-|z|^2}$ and $\Phi(z)=-z/\sqrt{1+|z|^2}$, it follows that
$$
\nabla f(x')=\Psi\left(\frac{-\nabla f(x')}{\sqrt{1+|\nabla f(x')|^2}}\right)=\Psi(\pi(\nu(x',f(x'))))
\quad
\text{for $x'\in U$.}
$$
This proves Lipschitz continuity of $\nabla f$ in $U$, as a composition of Lipschitz functions.
The only issue could be the Lipschitz continuity of $\Psi$: it is a smooth function defined for $|z|<1$, but it is unbounded. However, this does not cause any problems here, because
$$
\left|\frac{-\nabla f(x')}{\sqrt{1+|\nabla f(x')|^2}}\right|\leq \frac{M}{\sqrt{1+M^2}}<1.
$$
\end{proof}
\begin{proof}[Proof of Theorem~\ref{T1}]
Let $K\subset\R^n$ be a convex body. According to Lemma~\ref{T6}, for every $\eps>0$ there is $\delta_o>0$ such that for any $\delta\in (0,\delta_o)$, $\HH^{n-1}(\partial K\setminus\partial K(\delta))<\eps/2$. 
Since $K(\delta)\subset K$, 
it is easy to see that $\pi_{K(\delta)}(\partial K)=\partial K(\delta)$, so 
$\pi_{K(\delta)}(\partial K\setminus\partial K(\delta))=\partial K(\delta)\setminus\partial K$
and Lemma~\ref{T15} yields $\HH^{n-1}(\partial K(\delta)\setminus \partial K)\leq \HH^{n-1}(\partial K\setminus\partial K(\delta))$. Therefore, $\HH^{n-1}(\partial K\triangle\partial K(\delta))<\eps$.
Since the boundary of $K(\delta)$ is of class $C^{1,1}$ by Lemma~\ref{T3}, $W:=K(\delta)$ satisfies the claim of the theorem.
\end{proof}
The next result is a direct consequence of Theorem~\ref{T1} and it is a version of Theorem~\ref{T2}. 
We will use Corollary~\ref{T14} in the proofs of Theorems~\ref{T7}, \ref{T2}, and~\ref{global theorem}.
\begin{corollary}
\label{T14}
Let $f:\R^n\to\R$ be a convex function. Then for every $R>0$ and $\eps>0$, there is a convex function $g\in C^{1,1}(B^n(0,R))$ such that $g\geq f$ and
\begin{equation}
\label{eq11}
|\{ x\in B^n(0,R):\, f(x)\neq g(x)\}|<\eps.
\end{equation}
\end{corollary}
\begin{proof}
Let $M:=\sup_{\overbar{B}^n(0,2R)} f(x)$ and define
$$
W:=\{(x,y)\in\overbar{B}^n(0,2R)\times\R:\, f(x)\leq y\leq M+2R\}.
$$
That is, $W$ is an $(n+1)$-dimensional convex body bounded by the graph of $f$, the cylinder $\partial B^n(0,2R)\times\R$ and the hyperplane $y=M+2R$. According to Lemma~\ref{T6}, there is $\delta<R$ such that
$$
\HH^n(\partial W\setminus\partial W(\delta))<\eps.
$$
Since $W(\delta)$ is the union of closed balls of radius $\delta<R$ that are contained in $W$, it follows that
$$
\overbar{B}^n(0,2R)\times\{M+R\}\subset W(\delta),
$$
i.e., the intersection of $W(\delta)$ with the hyperplane $y=M+R$ is an $n$-dimensional closed ball of radius $2R$. Thus, if $\pi:\R^{n+1}\to\R^n$ is the orthogonal projection, $\pi(W(\delta))=\overbar{B}^n(0,2R)$, and hence for $x\in \overbar{B}^n(0,2R)$, we can define
$$
g(x):=\inf\{y:\, (x,y)\in W(\delta)\}.
$$
That is, the function $g:\overbar{B}^n(0,2R)\to\R$ parametrizes the bottom part of the boundary of $W(\delta)$. According to Lemma~\ref{T3}, the boundary of $W(\delta)$ is of class $C^{1,1}$ so $g\in C^{1,1}_{\rm loc}(B^n(0,2R))$ and hence $g$ is a convex function in $C^{1,1}(B^n(0,R))$.
Since $W(\delta)$ is contained in $W$ and hence in the epigraph of $f$, it follows that $g\geq f$.

Observe that
$$
\{x\in {B}^n(0,R):\, f(x)\neq g(x)\}\subset\pi(\partial W\setminus\partial W(\delta))
$$
and hence
$$
|\{x\in B^n(0,R):\, f(x)\neq g(x)\}|\leq
|\pi(\partial W\setminus\partial W(\delta))|\leq
\HH^{n}(\partial W\setminus\partial W(\delta))<\eps,
$$
because the orthogonal projection does not increase the Hausdorff measure and $\HH^n$ coincides with the Lebesgue measure in $\R^n$.
\end{proof}

\section{Proof of Theorem~\ref{T7}}
\label{S4}
\begin{lemma}
\label{T11}
Suppose that $f,g:B^n(0,R)\to\R$ are convex, $f\leq g$, and $g\in C^{1,1}(B^n(0,R))$. 
Then for almost all $x_o\in \{f=g\}$ we have
\begin{equation}
\label{eq16}
f(x)=f(x_o)+Df(x_o)(x-x_o)+\frac{1}{2}(x-x_o)^TD^2g(x_o)(x-x_o)+o(|x-x_o|^2).
\end{equation}
\end{lemma}
\begin{remark}
Note that $D^2g(x_o)$ in \eqref{eq16} is not a typo.
Also, we do not need the assumption that $g$ is convex or $C^{1,1}$. With a small modification, the proof works
under the assumption that $f\leq g\in C^1$ and $Dg$ is differentiable at $x_o$.
\end{remark}
\begin{proof}
It follows from Lemma~\ref{T19} that $f$ is differentiable at every point of the set $\{f=g\}$ and that $Df=Dg$ in $\{f=g\}$. Since $Dg$ is Lipschitz continuous, $Dg$ is differentiable a.e. by Rademacher's theorem. Therefore, it suffices to prove the result whenever $x_o\in\{f=g\}$ is a density point of that set and $Dg$ is differentiable at $x_o$. 

To simplify notation, without loss of generality, we may assume that $x_o=0$, and we need to prove that
$$
f(x)-f(0)-Df(0)x-\frac{1}{2}x^TD^2g(0)x=o(|x|^2).
$$
Since $f(0)=g(0)$ and $Df(0)=Dg(0)$, 
the left hand side equals
$$
(f(x)-g(x))+\left(g(x)-g(0)-Dg(0)x-\frac{1}{2}x^TD^2g(0)x\right)=(f(x)-g(x))+o(|x|^2).
$$
We used here the fact that $g$ is twice differentiable at $0$ (Taylor's theorem with the Peano remainder).
Thus it remains to show that $g(x)-f(x)=o(|x|^2)$.

Since $0$ is a density point of the set $\{ f=g\}$, for any $x$ we can find $y\in \{ f=g\}$ such that $|x-y|=o(|x|)$. 
For if not, there is $\eps>0$ and $x_k\to 0$ such that $B(x_k,\eps |x_k|)\cap\{f=g\}=\varnothing$ and that contradicts the fact that $0$ is a density point of $\{f=g\}$.

Clearly, $f(y)=g(y)$ and $Df(y)=Dg(y)$ by Lemma~\ref{T19}.  Therefore,
$$
f(x)\geq f(y)+Df(y)(x-y)=g(y)+Dg(y)(x-y), 
$$
where the inequality is a consequence of convexity of $f$. 
Since $f\leq g$, the above inequality and \eqref{eq14} yield
$$
0\leq g(x)-f(x)\leq g(x)-g(y)-Dg(y)(x-y)\leq M|x-y|^2=o(|x|^2).
$$
The proof is complete. 
\end{proof}
\begin{proof}[Proof of Theorem~\ref{T7}]
Let $f:\R^n\to\R$ be convex. Let $R>0$ and $\eps>0$ and let $g$ be as in Corollary~\ref{T14}. It follows from Lemma~\ref{T11} that for almost all $x\in \{f=g\}$, \eqref{eq3} is satisfied with $D^2f(x):=D^2g(x)$. Hence \eqref{eq3} holds true in $B(0,R)$ outside a set of measure less than $\eps$. Since it is true for any $R>0$ and $\eps>0$, it follows that \eqref{eq3} is satisfied almost everywhere.
\end{proof}

\section{Proof of Theorem~\ref{T8}}
\label{S5}

If $f$ is twice differentiable at $0$ as in \eqref{eq3}, then we have
$$
f(x)=f(0)+Df(0)x+\frac{1}{2}x^TD^2f(0)x+R(x)=
f(0)+Df(0)x+\langle Ax,x\rangle +R(x),
$$
where $A=\frac{1}{2}D^2f(0)$ and $R(x)=o(|x|^2)$.
Note that
$$
a(r):=\sup_{0<|x|\leq 2r}\frac{|R(x)|}{|x|^2}\to 0
\qquad
\text{as $r\to 0^+$.}
$$
Moreover,
$$
|R(x)|\leq a\Big(\frac{|x|}{2}\Big)\, |x|^2\leq a(|x|)|x|^2.
$$
\begin{proof}[Proof of Theorem~\ref{T8}]
Let $f$ be twice differentiable at $x$ as in \eqref{eq3}. We need to prove \eqref{eq13}. Without loss of generality we may assume that $x=0$, and hence we need to prove that
$$
\lim_{x\to 0}\frac{\sigma_x-Df(0)-D^2f(0)x}{|x|}=0
\quad
\text{for any } \sigma_x\in \partial f(x).
$$
For $x,y\neq 0$, we have
$$
f(x)=f(0)+Df(0)x+\langle Ax,x\rangle +R(x),
\quad
f(y)=f(0)+Df(0)y+\langle Ay,y\rangle +R(y).
$$
Since $f(x)+\langle \sigma_x,y-x\rangle\leq f(y)$, we have
$$
\langle\sigma_x,y-x\rangle
\leq 
f(y)-f(x)=
Df(0)(y-x)+\langle A(x+y),y-x\rangle+R(y)-R(x).
$$
We used here the fact that $A$ is symmetric and hence $\langle Ax,y\rangle=\langle Ay,x\rangle$.
Let
$$
y=x+w,
\quad
\text{where}
\quad
w=\sqrt{a(|x|)}\,|x|z,\ |z|=1.
$$
Then
$$
\langle\sigma_x,w\rangle\leq Df(0)w+
\langle A(2x+w),w\rangle +R(y)-R(x),
$$
$$
\langle\sigma_x-Df(0)-2Ax,w\rangle\leq
\langle Aw,w\rangle +R(y)-R(x).
$$
If $|x|$ is sufficiently small, then $a(|x|)\leq 1$ and hence $|w|\leq |x|$, so $|y|\leq 2|x|$. Therefore,
$$
|R(y)|\leq a\Big(\frac{|y|}{2}\Big)\, |y|^2\leq 4a(|x|)|x|^2,
\qquad
|R(y)-R(x)|\leq 5a(|x|)|x|^2.
$$
Taking the supremum over all $z$ with $|z|=1$ we get
$$
|\sigma_x-Df(0)-2Ax|\sqrt{a(|x|)}|x|\leq |A|a(|x|)|x|^2+5a(|x|)|x|^2,
$$
and hence
$$
\frac{|\sigma_x-Df(0)-2Ax|}{|x|}\leq (|A|+5)\sqrt{a(|x|)}\to 0
\quad
\text{as $x\to 0$.}
$$
Since $2A=D^2f(0)$, the result follows.
\end{proof}

\section{Proof of Theorem~\ref{T2}}
\label{S6}
One of the differences between Corollary~\ref{T14} and Theorem~\ref{T2} is that the function $g$ in Corollary~\ref{T14} is defined on the ball $B^n(0,R)$ only and the main step in the proof of Theorem~\ref{T2} will be to show that the function $g$ can be extended from a ball $B^n(0,R-\delta)$ to a convex function of class $C^{1,1}(\R^n)$. We will do it by gluing the function $g$ with a quadratic function of the form $a|x|^2-b$ and we need to know how to glue convex functions while maintaining their smoothness.

The maximum of two convex functions
$$
\max\{u,v\}=\frac{u+v+|u-v|}{2}
$$
is convex, but even if $u,v\in C^\infty$, the maximum $\max\{u,v\}$ need not be $C^1$. To overcome this difficulty, we will use the so called smooth maximum that was introduced in \cite{Azagra}.

Let $\theta\in C^\infty(\R)$ be such that $\theta(t)=|t|$ if and only if $|t|\geq 1$, $\theta$ is convex, $\theta(t)=\theta(-t)$ for all $t$, and $1$-Lipschitz.

It easily follows that $\theta(t)>0$ for all $t$ and $|\theta'(t)|<1$ if and only if $|t|<1$. Then, we define the {\em smooth maximum} function $\MM:\R^2\to\R$ as,
$$
\MM(x,y):=\frac{x+y+\theta(x-y)}{2}.
$$
It is easy to see that $\MM$ is smooth, convex and 
\begin{equation}
\label{eq20}
\MM(x,y)=\max\{x,y\} 
\quad
\text{whenever} 
\quad
|x-y|\geq 1.
\end{equation}
It is also not difficult to prove that $\MM(x,y)$ is non-decreasing in $x$ and $y$, because partial derivatives of $\MM$ are non-negative, see \cite[Lemma~2.1(viii)]{Azagra}. This observation and convexity of $\MM$ yield (see \cite[Proposition~2.2(i)]{Azagra})
\begin{lemma}
\label{T16}
If $u,v:U\to\R$ are convex functions defined in an open convex set $U\subset\R^n$, then $\MM(u,v):U\to\R$ is convex.
\end{lemma}
It is also obvious that if $u,v\in C^{1,1}_{\rm loc}(U)$, then $\MM(u,v)\in C^{1,1}_{\rm loc}(U)$.

We will use the smooth maximum to prove the following extension result.
\begin{proposition}
\label{T17}
Let $h\in C^{1,1}_{\rm loc}(B^n(0,R))$ be a convex function. Then, for every $r\in (0,R)$, there is a convex function $H\in C^{1,1}(\R^n)$, such that
\begin{equation}
\label{eq9}
H(x)=h(x)
\quad
\text{whenever } |x|\leq r.
\end{equation}
\end{proposition}
\begin{remark}
If $h\in C^k$, $k\in\mathbb{N}\cup\{\infty\}$, then $H\in C^k(\R^n)$. The proof remains the same.
\end{remark}
\begin{proof}
Choose $\rho\in (r,R)$ and let
$$
m:=\inf_{|x|\leq r} h,
\qquad
M:=\sup_{|x|=\rho}h.
$$
Then, we can find $a,b>0$ such that the function $q(x):=a|x|^2-b$ satisfies
\begin{equation}
\label{eq18}
q(x)<m-1
\qquad
\text{if } |x|\leq r
\end{equation}
\begin{equation}
\label{eq19}
q(x)>M+1
\qquad
\text{if } |x|=\rho,
\end{equation}
and we define
$$
H(x):=
\begin{cases}
\MM(h(x),q(x)) & \text{if } |x|\leq \rho,\\
q(x)           & \text{if } |x|>\rho.
\end{cases}
$$
It follows from \eqref{eq18} that $h(x)>q(x)+1$ if $|x|\leq r$, so  by \eqref{eq20}, we have $H(x)=\MM(h(x),q(x))=h(x)$ if $|x|\leq r$ and the condition \eqref{eq9} is satisfied. 
It follows from \eqref{eq19} that there is $\eps>0$ such that $q(x)>h(x)+1$ if $\rho\leq |x|\leq \rho+\eps$ and hence by \eqref{eq20}, $\MM(h(x),q(x))=q(x)$ when $\rho\leq |x|\leq \rho+\eps$.
Therefore, the convex functions $q(x)\in C^{1,1}(\R^n)$ and $\MM(h(x),q(x))\in C^{1,1}_{\rm loc}(B^n(0,R))$ coincide in the annulus $\rho\leq |x|\leq \rho+\eps$ and hence $H$ is convex in $\R^n$ with $H\in C^{1,1}_{\rm loc}(\R^n)$. Since $H=q\in C^{1,1}$ outside the compact ball $\overbar{B}^n(0,\rho)$, it follows that $H\in C^{1,1}(\R^n)$.
\end{proof}

\begin{proof}[Proof of Theorem~\ref{T2}]
Let $R>0$ be such that $|A\setminus \overbar{B}^n(0,R)|<\eps/2$. According to Corollary~\ref{T14} there is a convex function $\widetilde{g}\in C^{1,1}(B^n(0,2R))$ such that
$$
|\{x\in B^n(0,2R):\, f(x)\neq \widetilde{g}(x)\}|<\frac{\eps}{2}.
$$
Now, Proposition~\ref{T17} yields a convex function $g\in C^{1,1}(\R^n)$ such that $g(x)=\widetilde{g}(x)$ for $x\in \overbar{B}^n(0,R)$ and we have
$$
|\{x\in A:\, f(x)\neq g(x)\}|\leq
|A\setminus\overbar{B}^n(0,R)|+
|\{x\in B^n(0,R):\, f(x)\neq\widetilde{g}(x)\}|<\eps.
$$
\end{proof}

\section{Proofs of Theorems~\ref{global theorem}, \ref{corollary for convex hypersurfaces} and~\ref{global theorem for arbitrary U}}
\label{S7}

This section is less self-contained  than the others. The proof of the implication $(1)\Rightarrow(2)$ in Theorem~\ref{global theorem}
is easy and we will not show it here; see \cite[Proposition 1.10 and Theorem 2.5]{AzagraH}. The implication $(2)\Rightarrow(1)$ in Theorem~\ref{global theorem} is equivalent to the following result.
\begin{theorem}
\label{global theorem for coercive}
Let $f:\R^n\to\R$ be a convex function such that $\lim_{|x|\to\infty}f(x)=+\infty$. Then for every $\varepsilon>0$ there exists a convex function $g:\R^n\to\R$ of class $C^{1,1}_{{\rm loc}}(\R^n)$ such that $g\geq f$ and $|\{x\in \R^n : f(x)\neq g(x)\}|<\varepsilon$. 
\end{theorem}
Next, we give a proof of Theorem \ref{global theorem for coercive} that greatly simplifies the one provided by \cite{AzagraH}. Its main ingredients are Corollary~\ref{T14} above and the following lemma (whose elementary proof can be found in \cite[Lemma 5.3]{Azagra2}, which in turn is a refinement of the result of \cite{KirchheimK}).
\begin{lemma}
\label{regularity of convex envelope}
Let $\varphi:\R^n\to\R$ be a continuous function such that 
$
\lim_{|x|\to\infty}\varphi(x)=+\infty
$
and such that for every $R>0$ there exists $C_R>0$ so that for every $x, h\in B^n(0,R)$ we have
$$
\varphi(x+h)+\varphi(x-h)-2\varphi(x)\leq C_R|h|^2.
$$
Then the function $F=\operatorname{conv}(\varphi)$ has a similar property: for every $R>0$ there exists $C'_R>0$ such that for every $x, h\in B^n(0,R)$ we have
$$
F(x+h)+F(x-h)-2F(x)\leq C'_R |h|^2.
$$
Therefore $F\in C^{1,1}_{\rm loc}(\R^n)$.
\end{lemma}
Here $\operatorname{conv}(\varphi)$ denotes the convex envelope of $\varphi$, defined as the supremum of all convex functions less than or equal to $\varphi$.
\begin{proof}[Proof of Theorem \ref{global theorem for coercive}.]
By Corollary~\ref{T14}, for every $k\in\N$ we can find a convex function $g_k\in C^{1,1}(B^n(0,2k))$ such that $f\leq g_k$ and
$$
|\{ x\in B^n(0,2k):\, f(x)\neq g_k(x)\}|<\eps/2^k.
$$
For every $k\in\N$, let $\theta_k:(k-2, k+1)\to [0,\infty)$ be a $C^{\infty}$ convex function such that: 
\begin{enumerate}
\item $\theta_k(t)=0$ iff $k-1\leq t\leq k$;
\item $\lim_{t\to (k-2)^{+}}\theta_k(t)=+\infty$, and
\item $\lim_{t\to (k+1)^{-}}\theta_k(t)=+\infty$.
\end{enumerate}
Define $\varphi_k:\R^n\to (-\infty, +\infty]$ by 
$$
\varphi_k(x)=g_k(x)+\theta_k(|x|) \textrm{ if } k-2<|x|<k+1, \textrm{ and } \varphi_k(x)=+\infty \textrm{ otherwise.}
$$
When $n\geq 2$ and $k\geq 2$, the function $\varphi_k$ is not convex, but we do not need it to be.

Note that $\varphi_k(x)=g_k(x)$ on the annulus $A_k:=\{x: k-1\leq |x|\leq k\}$ (or ball in the special case $k=1$), and consider $\varphi:\R^n\to\R$ defined by
$
\varphi(x)=\inf_{k\in\N}\varphi_k(x).
$
It is clear that 
$$
f\leq \varphi \textrm{ on } \R^n, \textrm{ and }
\varphi\leq g_k \textrm{ on } A_k, \textrm{ for each } k\in\N,
$$
and in particular 
$
\lim_{|x|\to\infty}\varphi(x)=+\infty,
$
though $\varphi$ is finite everywhere. As a matter of fact, it is easily seen that $\varphi$ is locally the minimum of at most three continuous functions, and therefore it is continuous on $\R^n$. More precisely, for each $k\in\N$ we have that
$$
\varphi(x)=\min\{\varphi_{k-1}(x), \varphi_k(x), \varphi_{k+1}(x)\} \textrm{ for every } x\in A_k.
$$
Moreover, since $\lim_{|x|\to k^{-}}\varphi_{k-1}(x)=+\infty=\lim_{|x|\to k^{+}}\varphi_{k+2}(x)$ and $\varphi_k$ and $\varphi_{k+1}$ are bounded and $C^{1,1}$ on a neighborhod of the sphere $\{x: |x|=k\}$, there exist some $M_k, \delta_k>0$ such that $\varphi(x)=\min\{\varphi_{k}(x), \varphi_{k+1}(x)\}$ and 
$
\varphi_j(x+h)+\varphi_j(x-h)-2\varphi_j(x)\leq M_k|h|^2
$ 
for all $k-\delta_k\leq |x|\leq k+\delta_k$, $|h|\leq\delta_k$, and $j=k, k+1$. 
These inequalities easily follow from \eqref{eq14}.
This implies that
$$
\varphi(x+h)+\varphi(x-h)-2\varphi(x)\leq M_k|h|^2
$$
for all $k-\delta_k\leq |x|\leq k+\delta_k$, and $|h|\leq\delta_k$.
Similarly, there exist $M_k', \delta_k'>0$ such that $\varphi(x)=\min\{\varphi_{k-1}(x), \varphi_{k}(x), \varphi_{k+1}(x)\}$ and
$
\varphi_j(x+h)+\varphi_j(x-h)-2\varphi_j(x)\leq M_k'|h|^2
$
for all $k-1+\delta_{k-1}-\delta_k'\leq |x|\leq k-\delta_k+\delta_k'$, $|h|\leq\delta_k'$, and $j=k-1, k, k+1$, implying that
$$
\varphi(x+h)+\varphi(x-h)-2\varphi(x)\leq M_k'|h|^2
$$ 
for all $k-1+\delta_{k-1}-\delta_k'\leq |x|\leq k-\delta_k+\delta_k'$, and $|h|\leq\delta_k'$.
Since every ball is contained in a finite union of sets $A_k$, these estimates imply that for every $R>0$ there exist $C_R>0$ and $\delta_R>0$ so that for every $x\in B^n(0,R)$ and $|h|<\delta_R$ we have
\begin{equation}
\label{eq21}
E_h(x):=\varphi(x+h)+\varphi(x-h)-2\varphi(x)\leq C_R|h|^2.
\end{equation}
On the other hand, for $\delta_R\leq |h|\leq R$, we obviously have 
$
E_h(x)\leq 4M \leq\widetilde{C_R}|h|^2,
$
where $M:=\sup_{z\in B(0, 2R)}\varphi(z)$ and $\widetilde{C_R}=4M/\delta_R^2$. So by replacing $C_R$ with $\max\{C_R, \widetilde{C_R}\}$ we certainly have
\begin{equation}
\label{eq22}
\varphi(x+h)+\varphi(x-h)-2\varphi(x)\leq C_R|h|^2
\quad
\text{for all } x,h\in B^n(0,R).
\end{equation}

Therefore, \eqref{eq22} and Lemma \ref{regularity of convex envelope}, imply that the function $g:=\operatorname{conv}(\varphi)$ is of class $C^{1,1}_{\rm loc}$ (and it obviously satisfies $f\leq g\leq\varphi$).
Since $|\{ x\in B^n(0,2k):\, f(x)\neq g_k(x)\}|<\eps/2^k$ and $f\leq g\leq \varphi\leq g_k$ on $A_k$, it follows that
$
|\{ x\in A_k :\, f(x)\neq g(x)\}|<\eps/2^k
$
for every $k\in\N$, which implies that $|\{x\in \R^n : f(x)\neq g(x)\}|\leq \varepsilon$. 
\end{proof}

\begin{proof}[Proof of Theorem~\ref{corollary for convex hypersurfaces}.]
For the proof of the implication (1)$\Rightarrow$(2), see~\cite[Corollary 1.13]{AzagraH}. Regarding the implication (2)$\Rightarrow$(1),
the same proof as in \cite[Corollary 1.13]{AzagraH} gives us a (possibly unbounded) convex body $W_{\eps}$ with boundary $S_{\varepsilon}$ such that $W_{\varepsilon}=\frac{1}{t_0} g^{-1}(-\infty, t_0]$ for some $t_0\in (1, 2)$ and some convex function $g\in C^{1,1}_{\rm loc}(\R^n)$ such that $\mathcal{H}^{n-1}(S\setminus S_{\eps})<\eps/2$, $S=\partial W$, and $\mu\leq g$, where $\mu$ is the Minkowski functional of $W$, hence $W_{\eps}\subset W$. 

Now it suffices to show that $\pi_{W_\eps}(S)=S_\eps$,\footnote{This is very easy if $W$ is bounded and we used this fact in the proof of Theorem~\ref{T1}.}, where $\pi_{W_\eps}$ is the nearest point projection defined in \eqref{eq17}, because this fact and Lemma~\ref{T15} will imply
$
\mathcal{H}^{n-1}(S_{\eps}\setminus S)\leq \mathcal{H}^{n-1}(S\setminus S_{\eps}),
$
and hence $\mathcal{H}^{n-1}\left(S_{\eps}\triangle S\right)<\eps$.

Therefore, it remains to show that if $x\in S_\eps$, then there is $z\in S$ such that $\pi_{W_\eps}(z)=x$. Let $\nu(x)$ be the unit outward normal to $S_\eps$ at $x$. It suffices to show that the ray
$
R_x:=\{x+t\nu(x):\, t\geq 0\}
$
intersects $S$ at some point $z$, because clearly, $\pi_{W_\eps}(z)=x$. Suppose to the contrary that $R_x$ does not intersect with $S$ i.e. $R_x\subset\operatorname{int} W$.
The tangent hyperplane to $S_\eps$ at $x$ is defined by
$
T_x:=\{x+v:\langle v,\nu(x)\rangle=0\} 
$
and clearly $S_\eps\cap F_x=\varnothing$, where
$
F_x:=\{x+v:\langle v,\nu(x)\rangle>0\}
$
is an open half-space bounded by $T_x$.
Since $x\in\operatorname{int}W$, there is $\delta>0$ such that $D_{2\delta}\subset\operatorname{int}W$, where 
$
D_{2\delta}:=\{x+v\in T_x:\, |v|<2\delta\}
$
is the ball in $T_x$ centered at $x$ and of radius $2\delta$.
Since $D_{2\delta}\cup R_x\subset\operatorname{int}W$, it follows from the convexity of $W$ that $C_x\subset \operatorname{int} W$, where
$
C_x:=\{p+t\nu(x):\, p\in \partial D_\delta,\ t>0\}
$
is the side surface of a half-cylinder. 

Since $S$ does not contain any line, it follows that $W$ does not contain any line (cf.\ the argument at the beginning of the proof of Theorem~\ref{global theorem for arbitrary U}).
Therefore, for any $p\in R_x$, and any unit vector $v$ parallel to $T_x$, the lines
$
L_{p,v}:=\{p+tv:\, t\in\R\}
$
must intersect $S$ at least at one point. Denote all such points in $S$ by $A$. Since $A\subset F_x$, $F_x\cap S_\eps=\varnothing$, we have that $A\subset S\setminus S_\eps$. It is easy to see that $\HH^{n-1}(A)=\infty$ and hence $\HH^{n-1}(S\setminus S_\eps)=\infty$ which is a contradiction. To show that $\HH^{n-1}(A)=\infty$, note that the radial projection $\pi$ of $A$ onto $C_x$ along lines $L_{p,v}$ is $1$-Lipschitz and hence $\HH^{n-1}(A)\geq \HH^{n-1}(\pi(A))$. Now, for each two antipodal points in each sphere of radius $\delta$ in $C_x$ that is parallel to $T_x$, at least one belongs to $\pi(A)$. The mapping $\Phi:C_x\to C_x$ that maps points in $C_x$ to antipodal points is an isometry of $C_x$. Hence $\HH^{n-1}(\pi(A))=\HH^{n-1}(\Phi(\pi(A))$. Therefore, 
$$
2\HH^{n-1}(A)\geq 2\HH^{n-1}(\pi(A))\geq \HH^{n-1}(\pi(A)\cup \Phi(\pi(A)))=\HH^{n-1}(C_x)=\infty.
$$
The proof is complete.
\end{proof}

\begin{proof}[Proof of Theorem~\ref{global theorem for arbitrary U}] 
If an unbounded convex body $V$ contains a line $L$, then since $V$ is closed, it is easy to see that $V$ contains all lines parallel to $L$ that intersect with $V$. In particular $\partial V$ is the union of lines parallel to $L$.

\noindent
(2)$\Rightarrow$(1). 
Let us define $V_f$ as the closure of the epigraph of $f$, then $V_f$ is an unbounded convex body in $\R^{n+1}$. Since the graph of $f$ does not contain a line, $\partial V_f$ does not contain any line and we may apply
Theorem~\ref{corollary for convex hypersurfaces} and Remark~\ref{R2} to find a $C^{1,1}_{\textrm{loc}}$ convex body $W\subseteq V_f$ such that $\mathcal{H}^{n}(\partial V_f\setminus \partial W)<\varepsilon$. 
Since the projection $\pi:\R^n\times\R\to\R^n$ is 1-Lipschitz, it follows that $g(x):=\inf\{y: (x,y)\in W\}$ is a $C^{1,1}_{\textrm{loc}}$ convex function such that $g\geq f$, $|\{x\in U: g(x)\neq f(x)\}|<\varepsilon$. Note that $g$ is finite on all of $U$. Indeed, let $A:=\{x\in U: g(x)<\infty\}$; we understand that $g(x)=\infty$ if the line $\{(x,t): t\in\R\}$ does not intersect $W$. The set $A$ is convex because $W$ is convex. If $A\neq U$ then, for some $x_0\in U$, some number $c$ and some linear function $\ell:\R^n\to\R$, we have $\ell(x_0)\geq c\geq \ell(x)$ for all $x\in A$, implying that for all $x\in U\cap\ell^{-1}(c, \infty)$ the vertical line $\{(x,t): t\in\R\}$ does not intersect $W$. But it is easy to see that the set $\partial V_f\cap \{(x,t) : x\in \overbar{U}\cap\ell^{-1}(c, \infty)\}$ has infinite Hausdorff $n$-dimensional measure, therefore it must intersect $\partial W$. Hence $g(x)=f(x)<\infty$ for some $x\in U\cap\ell^{-1}(c, \infty)$, a contradiction.

\noindent
(1)$\Rightarrow$(2). Suppose to the contrary  that $f$ satisfies (1) and that the graph of $f$ contains a line $L$. Then, as we observed above, the graph of $f$ is the union of lines parallel to $L$. Thus, $U$ is the union of lines parallel to $\pi(L)$ and clearly, $f$ is affine on each such a line. Now, an argument similar to the proof of \cite[Proposition 1.10]{AzagraH} yields that if $g:U\to\R$ is a convex function such that $|\{x : f(x)\neq g(x)\}|<\infty$, then $g=f$. Hence $g\not\in C^{1,1}_{\rm loc}$, because $f\not\in C^{1,1}_{\rm loc}$ and we arrive to a contradiction with (1).
\end{proof}

\end{document}